\documentclass{amsart}
\usepackage{amssymb,latexsym,amscd,amsthm,amssymb, amscd, amsfonts,enumerate,array,bbm,bm}
\usepackage[sorted,compressed-cites,sorted-cites,initials]{amsrefs}
\usepackage[cmtip,all]{xy}
\usepackage{hyperref,geometry}
\usepackage{pifont,mathrsfs,mathtools,stmaryrd}

\numberwithin{equation}{section}

\newtheorem{theorem}{Theorem}[section]
\newtheorem{proposition}[theorem]{Proposition}

\newtheorem{corollary}[theorem]{Corollary}

\newtheorem{lemma}[theorem]{Lemma}
\newtheorem{problem}[theorem]{Problem}

\theoremstyle{definition}

\DeclareMathOperator{\Mat}{Mat}

\def\FF{\mathbb{F}}
\def\ZZ{\mathbb{Z}}
\def\QQ{\mathbb{Q}}

\def\kk{\Bbbk}

\newcommand{\tensor}{\otimes}

\date{\today}

\thanks{The authors were partially supported
by the BSF grant no.~2012365 (A.~B. and D.~K.),
NSF grant~DMS-1403527 (A.~B.), the ERC grant no.~247049 (D.K.)
and the Simons foundation collaboration grant no.~245735~(J.~G.)
}

\begin{document}
\newgeometry{margin=2cm}

\title{Integrable clusters}
\author{Arkady Berenstein}
\address{Department of Mathematics, University of Oregon,
Eugene, OR 97403, USA} \email{arkadiy@math.uoregon.edu}

\author{Jacob Greenstein}
\address{Department of Mathematics, University of
California, Riverside, CA 92521.} 
 \email{jacob.greenstein@ucr.edu}

\author{David Kazhdan}
\address{\noindent Department of Mathematics, Heberew University, Jerusalem, Israel}
\email{kazhdan@math.huji.ac.il}

\begin{abstract}
The goal of this note is to study quantum clusters in which cluster variables (not coefficients) commute which each other. It turns out 
that this property is preserved by mutations. Remarkably, this is equivalent to the celebrated sign coherence conjecture recently 
proved by M.~Gross, P.~Hacking, S.~Keel and M.~Kontsevich in~\cite{GHKK}.
\end{abstract}

\maketitle

\section{Main results}
Let $\tilde B$ be an integer $m\times n$ matrix, $n\le m$, and let~$\Lambda=(\lambda_{ij})_{1\le i,j\le m}$ be a rational skew-symmetric 
 $m\times m$-matrix compatible with $\tilde B$ in the sense of~\cite{BZ}, 
 that is
\begin{equation}\label{eq:compat}
 \tilde B^T\Lambda=\begin{pmatrix}D&\mathbf 0\end{pmatrix}
\end{equation}
where $D$ is a rational invertible diagonal $n\times n$-matrix.
This defines, on the one hand, a seed $\Sigma=(\mathbf x,\tilde B)$, $\mathbf x=(x_1,\dots,x_m)$ and 
the upper cluster algebra $\mathcal U=\mathcal U(\Sigma)\subset \mathscr L_m=\QQ[x_1^{\pm 1},\dots,
x_m^{\pm 1}]$ and, on the other hand, the Poisson algebra structure on~$\mathscr L_m$ via 
$\{x_i,x_j\}_\Lambda=\lambda_{ij} x_i x_j$, $1\le i,j\le m$ so that $\mathcal U$ is a Poisson subalgebra of $\mathscr L_m$. This in turn defines a Poisson scheme $\mathscr X$ such that 
for any field extension~$\FF$ of~$\QQ$,
$\mathscr X_\FF=\operatorname{Spec}(\mathcal U\tensor_{\QQ} \FF)$.

We say that the seed $\Sigma=(\mathbf x,\tilde B)$ is {\em integrable} if 
$\{x_i,x_j\}_\Lambda=0$ for all~$1\le i,j\le n$ for some $\Lambda$ compatible with~$\tilde B$. 
It is easy to show (Lemma~\ref{lem:classical-principal-Lambda}) that if $\Sigma$ is principal (see Section~\ref{sec:Proofs})
then it is integrable and the matrix~$\Lambda$ is uniquely determined by~$\tilde B$ and~$D$.
We say that $\Sigma$ is {\em completely integrable} if $\Sigma$ and 
all its mutations are integrable.
This definition is justified by the following observation.
Let $m=2n$ and assume that $\Sigma$ is integrable and that $\Lambda$ is invertible, i.e., 
$\{\cdot,\cdot\}_\Lambda$ is a symplectic bracket. 
Then for any $H\in\QQ[x_1,\dots,x_n]\setminus\{0\}$ the triple $(\mathcal U,\{\cdot,\cdot\}_\Lambda,H)$ is an integrable system with Hamiltonian~$H$; 
the natural map $\pi_{\Sigma}:\mathscr X\to \mathbb A^n$ is a Lagrangian fibration (see~\cite{A}).

\begin{theorem}
Every principal seed is completely integrable.
\end{theorem}
\begin{corollary}
For any seed $\Sigma'=({\mathbf y},\tilde B')$, $\mathbf y=(y_1,\dots,y_m)$,  mutation equivalent to a principal integrable seed~$\Sigma$, the natural inclusion $\kk[y_1,\ldots,y_n]\subset {\mathcal U}(\Sigma')={\mathcal U}$ defines a Lagrangian fibration 
$\pi_{\Sigma'}:\mathscr X\to \mathbb A^n$.
\end{corollary}
This gives rise to the following natural problem.
\begin{problem}
Given mutation equivalent integrable seeds $\Sigma$, $\Sigma'$, describe the intersection $\pi^{-1}_\Sigma(c)\cap \pi^{-1}_{\Sigma'}(c')$ for generic points~$c,c'\in\mathbb A^n$
and determine for 
which~$\Sigma$, $\Sigma'$ 
this intersection is transversal. 
\end{problem}

It turns out that classical results carry over verbatim to the quantum case.
We 
say that a quantum seed $\Sigma=(\mathbf X,\tilde B)$, where $\mathbf X=(X_1,\dots,X_m)$ and $\tilde B$ is an $m\times n$ integer matrix
(see Section~\ref{sec:Proofs}),
is {\em integrable} if $X_iX_j=X_jX_i$  for all $1\le i,j\le n$. 
Then, similarly to the classical case, we say that a quantum seed $\Sigma$ is {\em completely integrable} if $\Sigma$ and 
all its mutations are integrable.
\begin{theorem}\label{thm:II}
Every integrable principal quantum seed is completely integrable.
\end{theorem}
By Lemma~\ref{lem:convert-to-int} any (quantum or classical) seed can be converted into a principal integrable one merely by duplicating the ambient (quantum) torus 
as in~\cite{BZ}*{Section~3}.

A {\em generalized quantum integrable system} is a pair $(\mathcal U,\iota)$ where~$\mathcal U$ is 
a $\kk$-algebra of Gelfand-Kirillov dimension~$2n$  
and~$\iota:\kk[x_1,\dots,x_n]\hookrightarrow \mathcal U$ is an embedding of algebras such 
that $\iota(\kk[x_1,\dots,x_n])$ is a maximal commutative subalgebra of $\mathcal U$. This defines a quantum  analogue of Lagrangian fibration, whose {\em quantum Lagrange fiber} over a maximal ideal~$\mathfrak m$ of~$\kk[x_1,\dots,x_n]$ is the 
left $\mathcal U$-module $\mathcal U_{\iota,\mathfrak m}:=
\mathcal U/\mathcal U\cdot\iota(\mathfrak m)$. 
Then the (quantum) intersection of  $\mathcal U_{\iota, \mathfrak m}$ and 
$\mathcal U_{\iota',\mathfrak m'}$ is the left $\mathcal U$-module $\mathcal U/\mathcal U\cdot(\iota(\mathfrak m)+\iota'(\mathfrak m'))$.

Following~\cite{BZ1}, given a field $\kk$ containing $\QQ(q^{\frac{1}{2}})$ and a {\em quantum seed}  $\Sigma=(\mathbf X,\tilde B)$, we denote by $\mathcal U=\mathcal U(\Sigma)$ its {\em upper quantum cluster algebra} (see Section~\ref{sec:Proofs} for the details). 

\begin{corollary}\label{cor:integrable system}
For any quantum seed $\Sigma'=(\mathbf Y,\tilde B')$ mutation equivalent to 
a given principal integrable quantum seed $\Sigma$, the natural inclusion $\iota:\kk[Y_1,\ldots,Y_n]\hookrightarrow {\mathcal U}(\Sigma')={\mathcal U}$ defines a generalized quantum integrable system $({\mathcal U},\iota)$. 
\end{corollary}

\section{Notation and proofs}\label{sec:Proofs}

We will only prove quantum results since their classical counterparts follow by specializing $q$ to~$1$.

Let~$\Lambda\in\Mat_{m\times m}(\ZZ)$ be compatible with~$\tilde B$ in the sense 
of~\eqref{eq:compat} where $D\in\Mat_{n\times n}(\ZZ)$ and has positive diagonal entries.
Following~\cite{BZ1},
we associate with the pair $(\Lambda,\tilde B)$ a {\em quantum seed} $\Sigma=(\mathbf X,\tilde B)$,
$\mathbf X=(X_1,\dots,X_m)\in (\mathcal F^\times)^m$ where $\mathcal F$ is a skew field, 
such that the subalgebra~$\mathscr L_{\mathbf X}$ of~$\mathcal F$ generated by~$\mathbf X$ over some central subfield~$\kk$ of~$\mathcal F$ 
containing~$\QQ(q^{\frac12})$
has presentation
$$
X_i X_j=q^{\lambda_{ij}} X_j X_i,\qquad 1\le i<j\le m.
$$
Given $1\le j\le n$, we define $\mu_j(\Sigma)=(\mathbf X',\mu_j(\tilde B))$ where 
$\mu_j(\tilde B)$ is the Fomin-Zelevinsky mutation of~$\tilde B$ from~\cite{FZ}*{} and $\mathbf X'$ is obtained from~$\mathbf X$ by 
replacing $X_j$ with 
\begin{equation}\label{eq:mutation}
X'_j=X^{[b_j]_+-e_j}+X^{[-b_j]_+-e_j},
\end{equation}
where $b_j$ is the $j$th column of~$\tilde B$, for each $a=(a_1,\dots,a_m)\in\mathbb Z^m$ we set $[a]_+=(\max(0,a_1),\dots,\max(0,a_m))$,
$$
X^a=q^{\frac12 \sum_{1\le i<j\le m}\lambda_{ji} a_i a_j} X_1^{a_1}\cdots X_m^{a_m}
$$
and $\{e_i\}_{1\le i\le m}$ is the standard basis of~$\mathbb Z^m$. After~\cite{BZ1}*{Section~2}, $\mu_j(\Sigma)$ is also a quantum seed and we refer
to it as the $j$th mutation of~$\Sigma$.
A quantum seed~$\Sigma'$ is {\em mutation equivalent} to $\Sigma$ if it can be obtained from~$\Sigma$ by a sequence of mutations.

Following~\cite{NZ}, 
we say that an $m\times n$-matrix $\tilde B$ is {\em sign-coherent} if for every $1\le j\le n$ there exists $\epsilon_j\in\{-1,1\}$
such that $\epsilon_j b_{ij}\ge 0$ for all $n+1\le i\le m$. We say that~$\tilde B$ is {\em totally sign-coherent}
if all matrices mutation equivalent to~$\tilde B$ are sign-coherent.
\begin{proposition}\label{prop:key prop}
Suppose $\Lambda\in\Mat_{m\times m}(\ZZ)$ and $\tilde B\in\Mat_{m\times n}(\ZZ)$ are compatible, $\tilde B$ is totally sign-coherent
and the corresponding quantum seed~$\Sigma$ is integrable.
Then $\Sigma$ is completely integrable.
\end{proposition}
\begin{proof}
We need the following 
\begin{lemma}\label{lem:key lem}
Suppose that $\Sigma=(\mathbf X,\tilde B)$ is integrable and $\tilde B$ is sign-coherent. Then 
$\mu_j(\Sigma)$, $1\le j\le n$, is integrable.
\end{lemma}
\begin{proof}
Since~$\tilde B$ is sign-coherent, either $X_j X^{[b_j]_+-e_j}$ or~$X_jX^{[-b_j]_+-e_j}$ is 
contained in $\kk[X_1,\dots,X_n]$. 
Since for every~$1\le i\not=j\le n$ we have 
$X_i X^{[b_j]_+-e_j}=q_{ij} X^{[b_j]_+-e_j}X_i$ and $X_i X^{[-b_j]_+-e_j}=q_{ij}X^{[-b_j]_+-e_j}X_i$ for some $q_{ij}\in\kk^\times$, it follows 
that $q_{ij}=1$. 
Then by~\eqref{eq:mutation} we have $X_i X'_j=X'_j X_i$ for all~$1\le i\not=j\le n$.
\end{proof}

We complete the proof by induction on the number of mutations applied to the initial seed~$\Sigma$. If~$\Sigma'=\mu_{j_{1}}\cdots\mu_{j_{k}}(\Sigma)=
\mu_{j_1}(\Sigma'')$ where  
$\Sigma''=(\mathbf X'',\tilde B'')=\mu_{j_{2}}\cdots\mu_{j_k}(\Sigma)$ is integrable by the induction hypothesis and $\tilde B''$ is sign-coherent by 
assumption on~$\tilde B$. It remains to apply the Lemma.
\end{proof}

\begin{proof}[Proof of Theorem~\ref{thm:II}] 
Recall from~\cite{FZ-CAIV}*{Remark~3.2} that $\tilde B$ is called principal if 
$\tilde B=\begin{pmatrix} B\\ I_n
         \end{pmatrix}$ where 
         $B\in\Mat_{n\times n}(\ZZ)$ and $DB$ is skew symmetric for some 
         $D\in\Mat_{n\times n}(\ZZ)$ diagonal with positive diagonal entries.
The following result was initially conjectured in~\cite{FZ-CAIV}*{Conjecture~5.4 and Proposition~5.6(iii)} (see also~\cite{NZ}*{(1.8)}).
\begin{lemma}[\cite{GHKK}*{Corollary~5.5}]
Any principal $\tilde B$ is totally sign-coherent.
\end{lemma}
Theorem~\ref{thm:II} is immediate from Proposition~\ref{prop:key prop} and the above Lemma.
\end{proof}

\begin{proof}[Proof of Corollary~\ref{cor:integrable system}]

We need the following obvious classification of  (quantum) integrable seeds with $m=2n$. 

\begin{lemma}\label{lem:classical-principal-Lambda}
\label{le:compatible} Let \begin{equation}\label{eq:block lambda}
\Lambda=\begin{pmatrix}\mathbf 0&\Lambda_1\\ -\Lambda_1^T&\Lambda_2\end{pmatrix},\qquad \tilde B=\begin{pmatrix} B\\ C
         \end{pmatrix}
\end{equation} 
with 
$B,C,\Lambda_1,\Lambda_2 \in\Mat_{n\times n}(\QQ)$ and~$\Lambda_2^T=-\Lambda_2$. Then \eqref{eq:compat} holds for some invertible diagonal $D\in\Mat_{n\times n}(\QQ)$
if and only if 
$$\det C\ne 0,\quad (DB)^T=-DB,\quad \Lambda_1=-DC^{-1},\quad \Lambda_2=-(C^{-1})^TDB C^{-1}.
$$
\end{lemma}

Furthermore, following~\cite{BZ1}, to each quantum seed $\Sigma$ one associates the {\em quantum upper cluster algebra}
$\mathcal U(\Sigma)=\bigcap_{1\le i\le n} \mathcal U_i$ where $\mathcal U_i$ is the subalgebra 
of $\mathscr L_{\mathbf X}$ generated by $\mathbf X$ and 
$X_i'$. As shown in~\cite{BZ1}*{Theorem~5.1}, $\mathcal U(\Sigma)=\mathcal U(\mu_j(\Sigma))$ for all~$1\le j\le n$.
We need the following Lemma.
\begin{lemma}
Let $\Sigma=(\mathbf X,\tilde B)$ be a  quantum integrable seed with $m=2n$. Then 
\begin{enumerate}[{\rm(a)}]
\item $\kk[ X_1^{\pm 1},\dots,X_n^{\pm 1}]$ is a maximal commutative subalgebra of~$\mathscr L_{\mathbf X}$;
\item $\kk[X_1,\dots,X_n]$ is a maximal commutative subalgebra of~$\mathcal U(\Sigma)$.
\end{enumerate}
\end{lemma}
\begin{proof} 
To prove (a) note that for each~$1\le i\le n$ the centralizer~$\mathscr C_i$ of~$X_i$ in~$\mathscr L_{\mathbf X}$ 
is the $\kk$-linear span of $\{X^a\,:\, \sum_{1\le j\le m} \lambda_{ij}a_j=0\}$. 
Since $\Lambda$ is as in \eqref{eq:block lambda} with $\det \Lambda_1\ne 0$, $\bigcap_{1\le i\le n}\mathscr C_i$ is spanned by all
monomials~$X^a$ with $a_j=0$, $j>n$. This implies that~$\kk[X_1^{\pm 1},\dots,X_n^{\pm 1}]$ is a maximal commutative subalgebra of~$\mathscr L_{\mathbf X}$. 

To prove~(b), denote 
$\mathcal U_i'=\mathcal U_i\cap \kk[X_1^{\pm 1},\dots,X_n^{\pm 1}]$. It is easy to see that~$\mathcal U'_i=\kk[X_1^{\pm 1},\dots,X_i,\dots, X_n^{\pm 1}]$. 
Since $\mathcal U\cap \kk[X_1^{\pm 1},\dots,X_n^{\pm 1}]=\bigcap_{1\le i\le n}\mathcal U'_i=\kk[X_1,\dots,X_n]$,
part~(b) is now immediate from part~(a).
\end{proof}
The Corollary~\ref{cor:integrable system} is now immediate because each quantum seed mutation equivalent to a given principal integrable one is automatically integrable by Theorem \ref{thm:II}.
\end{proof}

We conclude by showing that every (quantum) seed can be converted into a principal completely integrable one. Recall 
that in~\cite{BZ}*{Section~3} to every seed~$\Sigma=(\mathbf X,\tilde B)$ in~$\mathscr L_{\mathbf X}$ one associates 
a seed $\Sigma^\bullet=(\mathbf X^\bullet,\tilde B^\bullet)$ in the duplicated quantum torus $\mathscr L_{\mathbf X}^{(2)}=\bigoplus_{e,e'\in \ZZ^m} \kk X^{(e,e')}$ with $\Lambda^{(2)}=\begin{pmatrix}\Lambda &\mathbf 0 \\ 
                              \mathbf 0 &-\Lambda
															\end{pmatrix}$,  
as follows
$$
X^\bullet_i=X^{(e_i,e_i)},\qquad X^\bullet_{i+n}=X^{(b_i^{>n},-b_i^{\le n})},\qquad 1\le i\le n,\qquad \tilde B^\bullet=\begin{pmatrix}B\\I_n\end{pmatrix}
$$
where  we abbreviate $a^{>n}=\sum_{n<i\le m}a e_{i}$, $a^{\le n}=\sum_{1\le i\le n}a_i e_i$ for $a=(a_1,\dots,a_m)\in \ZZ^m$. By definition,
$\mathscr L_{\mathbf X}$ identifies with a subalgebra of~$\mathscr L_{\mathbf X}^{(2)}$ via $X^e\mapsto X^{(e,0)}$, $e\in\ZZ^m$. Then 
$\mathscr L_{\mathbf X}^{(2)}=\mathscr L_{\mathbf X}\cdot \mathscr C_{\mathbf X}$ and hence
$\mathcal U(\Sigma)\cdot  \mathscr C_{\mathbf X}=\mathcal U(\Sigma^\bullet)\cdot  \mathscr C_{\mathbf X}$ as subalgebras of $\mathscr L^{(2)}_{\bf X}$, 
where $\mathscr C_{\mathbf X}$ is the subalgebra of~$\mathscr L_{\mathbf X}^{(2)}$ generated by the $X^{(e_i,0)}$, $n<i\le m$ and 
by the $X^{(0,e_j)}$, $1\le j\le m$.

The following is immediate from \cite{BZ}*{(3.19)--(3.22) and Lemma~3.4}.
\begin{lemma}
The seed $\Sigma^\bullet$  is principal and integrable,
hence completely integrable.
\label{lem:convert-to-int}
\end{lemma}

\end{document}